\documentclass[11pt]{article}
\usepackage{amsmath}
\usepackage{amsthm,amsfonts,amssymb,epsfig,graphics}
\usepackage{pdfsync}
\usepackage[latin1]{inputenc}
\usepackage[english]{babel}
\usepackage{hyperref}

\def\homeo{hom\'eomorphisme}

\newcommand{\bbR}{{\mathbb{R}}}
\newcommand{\bbC}{{\mathbb{C}}}

\newcommand{\bbS}{{\mathbb{S}}}

    \def\cT{{\cal T}}
  \def\cI{{\cal I}} \def\cO{{\cal O}}

\def\homeo{\mathrm{Homeo}}

\def\?{$^{***}$\marginpar{?}}

\newtheorem{ques}{Question}
\newtheorem*{ques*}{Question}
\newtheorem*{prop*}{Proposition}
\newtheorem*{conj*}{Conjecture}
\newtheorem*{theo*}{Theorem}
\newtheorem{coro}{Corollary}[section]

\newtheorem{affi*}{Affirmation}

\newtheorem{lemm}[coro]{Lemma}

\newtheorem*{lemm*}{Lemma}
\def\?{\footnote{?}}

\newlength{\espaceavantspecialthm}
\newlength{\espaceapresspecialthm}
{\vskip \espaceapresspecialthm}

{\vskip \espaceapresspecialthm}

\title{On  closed subgroups of the group of homeomorphisms of a manifold}
\author{Fr\'ed\'eric Le Roux}

\begin{document}
\sloppy 

\maketitle

\begin{abstract}
Let $M$ be a  triangulable compact manifold. We prove that, among closed subgroups of 
$\homeo_{0}(M)$ (the identity component of the group of homeomorphisms of $M$), 
the subgroup consisting of volume preserving elements is maximal. 
\end{abstract}


\section{Introduction}

The theory of of groups acting on the circle is  very rich  (see in particular the monographs~\cite{Ghys01,Navas07}). 
The theory is far less developed in higher dimension, where it seems difficult to discover more than some isolated islands in a sea of chaos. In this note, we are interested in the closed subgroups of  the group  $\homeo_{0}(M)$,  the identity component of the group of homeomorphisms of some compact topological $n$-dimensional manifold $M$. We will show that, when $n \geq 2$, for any
\emph{good} (nonatomic and with total support) probability measure $\mu$, the subgroup of elements that preserve $\mu$ is maximal among closed subgroups.

Let us recall some related results in the case when $M$ is the circle. De La Harpe conjectured that $PSL(2,\bbR)$ is a maximal closed subgroup (\cite{Bestvina}). Ghys  proposed a list of closed groups acting transitively, asking whether, up to conjugacy,  the list was complete (\cite{Ghys01}); the list consists in the whole group, $SO(2)$, $PSL(2,\bbR)$, the group  $\homeo_{k,0}(\bbS^1)$ of elements that commutes with some rotation of order $k$, and the group $PSL_{k}(2,\bbR)$ which is defined analogously. The first conjecture was solved by Giblin and Markovic in~\cite{GiMa06}. These authors also answered Ghys's question affirmatively, under the additional hypothesis that the group contains some non trivial arcwise connected component.
Thinking of the two-sphere with these results in mind, one is naturally led  to the following questions.
\begin{ques}
Let $G$ be a proper closed subgroup of $\homeo_{0}(\bbS^2)$ acting transitively. Assume that $G$ is not a (finite dimensional) Lie group. Is $G$ conjugate to one of the two subgroups: (1) the centralizer of the antipodal map $x \mapsto -x$,  (2) the subgroup of area-preserving elements?
\end{ques}
Note that the centralizer of the antipodal map is the group of lifts of homeomorphisms of the projective plane; it is the spherical analog of the groups $\homeo_{k,0}(\bbS^1)$.
\begin{ques}
Is $PSL(2,\bbC)$ maximal among closed subgroups of $\homeo_{0}(\bbS^2)$?
\end{ques}

On the circle the group of measure-preserving elements coincides with $SO(2)$. It is not a maximal closed subgroup since it is included in $PSL(2,\bbR)$. In contrast, we propose to prove that the closed  subgroup of area-preserving homeomorphisms of the two-sphere is maximal. To put this into a general context, let $M$ be a compact topological manifold whose dimension is greater or equal to $2$. We  assume that $M$ is triangulable and (for simplicity) without boundary. Let us equip $M$ with a  probability measure $\mu$ which is assumed to be \emph{good}: this means that every finite set has measure zero, and every non-empty open set has positive measure. 
We consider the group $\homeo_{0}(M)$ of homeomorphisms of $M$ that are isotopic to the identity, and the subgroup $\homeo_{0}(M,\mu)$ of elements that preserve the measure $\mu$. According to the famous Oxtoby-Ulam theorem (\cite{OxUl41,GoPe75}, see also~\cite{fathi80}), if $\mu'$ is another good probability measure on $M$ then it is homeomorphic to $\mu$, meaning that there exists an element $h \in \homeo_{0}(M)$ such that $h_{*}\mu=\mu'$. In particular the subgroup $\homeo_{0}(M,\mu')$ is isomorphic to $\homeo_{0}(M,\mu)$.
We equip these transformation groups with the topology of uniform convergence, which turns them into topological groups. The subgroup $\homeo_{0}(M,\mu)$ is easily seen to be closed in $\homeo_{0}(M)$. Note that according to Fathi's theorem (first theorem in \cite{fathi80}), $\homeo_{0}(M,\mu)$ coincides with the identity component in the group of measure preserving homeomorphisms.
The aim of the present note is to prove the following.
\begin{theo*}
The group $\homeo_{0}(M,\mu)$ is maximal among closed subgroups of 
the group $\homeo_{0}(M)$. 
\end{theo*}

In what follows we consider some element $f \in \homeo_{0}(M)$ that does not preserves the measure $\mu$, and we denote by $G_{f}$ the subgroup of $\homeo_{0}(M)$ generated by 
$$\{f \} \cup  \homeo_{0}(M,\mu).$$
Our aim is to show that the group $G_{f}$ is dense in $\homeo_{0}(M)$.

\section{Localization}
In this section we show how to find some element in $G_{f}$ that has small support and contracts the volume of some given ball.

\paragraph{Good balls} A \emph{ball} is any subset of $M$ which is homeomorphic to a euclidean  ball in $\bbR^n$, where $n$ is the dimension of $M$. 
We will need to consider balls which are locally flat and whose boundary has measure zero. More precisely, let us denote by $B_{r}(0)$ the euclidean ball with radius $r$ and center $0$ in $\bbR^n$. A ball $B$ will be called \emph{good} if $\mu(\partial B)=0$ and if there exists a topological embedding (continuous one-to-one map) $\gamma : B_{2}(0) \to M$ such that $\gamma(B_{1}(0)) = B$. Note that, due to countable additivity, if $\gamma : B_{1}(0) \to M$ is any topological embedding, then for almost every $r \in (0,1)$ the ball $\gamma(B_{r}(0))$ is good.

\bigskip	

\paragraph{Oxtoby-Ulam theorem} We will need the following consequence of the Oxtoby-Ulam theorem. Let $B_{1}, B_{2}$ be two good balls in the interior of some manifold $M'$, with or without boundary (what we have in mind is either $M'=M$ or $M'$ is a euclidean ball). Let $\mu'$ be a good probability measure on $M'$ which assigns measure zero to the boundary $\partial M'$. Denote by $\homeo_{0}(M',\mu')$ the identity component of the group of homeomorphisms of $M'$ which are supported in the interior of $M'$ and preserve $\mu'$. Assume $\mu'(B_{1}) = \mu'(B_{2})$. Then \emph{there exists $\phi \in \homeo_{0}(M',\mu')$ such that $\phi(B_{1}) = B_{2}$}.
 To construct $\phi$, we first choose a good ball $B$ in the interior of $M'$ that contains $B_{1}, B_{2}$ in its interior. According to the annulus theorem (\cite{Kirby69,Quinn82}), we may find a homeomorphism $\phi'$ supported in the ball $B$ that sends $B_{1}$ onto $B_{2}$\footnote{One may probably avoid the use of the annulus theorem here, since the ball $B$ may be constructed explicitly by gluing the two good balls $B_{1}$ and $B_{2}$ to a piecewise linear tube connecting them.}. A first use of  the Oxtoby-Ulam theorem provides a homeomorphism $\phi_{1}$ supported in $B_{2}$ and sending the measure $(\phi'_{*}\mu')_{\mid B_{2}}$ to the measure $\mu'_{\mid B_{2}}$. A second use of the same theorem gives a homeomorphism $\phi_{2}$ supported in $B \setminus B_{2}$ and sending the measure $(\phi'_{*}\mu')_{\mid B \setminus B_{2}}$ to the measure $\mu'_{\mid B \setminus  B_{2}}$. Then $\phi$ is obtained as $\phi_{2} \phi_{1} \phi'$. Note that, since $\phi$ is supported in the ball $B$, Alexander's trick (\cite{Alexander23}) provides an isotopy from the identity to $\phi$ within the homeomorphisms of $B$ that preserves the measure $\mu'$, which shows that $\phi$ belongs to $\homeo_{0}(M',\mu')$.

\bigskip	

\paragraph{Triangulations} We will also need triangulations which have good properties with respect to the measure $\mu$. 
We begin with any triangulation $\cT$ of $M$. We would like the $(n-1)$-skeleton of $\cT$ to have measure zero, but some $(n-1)$-dimensional simplices may have positive measure. We fix this as follows. Each $n$-dimensional  simplex $s$ of $\cT$ is homeomorphic to the standard $n$-dimensional simplex ; let $\mu_{s}$ be a probability measure on $s$ which is the homeomorphic image of the Lebesgue measure on the standard simplex. The measure
$$
\mu' = \frac{1}{N} \sum \mu_{s}
$$
(where $N$ denotes the number of  $n$-dimensional  simplices of $\cT$) is a good probability measure on $M$ for which the $n-1$-dimensional simplices have measure zero.
 We apply the Oxtoby-Ulam theorem to get a homeomorphism $h$ of $M$ sending $\mu'$ to $\mu$. Then we consider the image triangulation $\cT_{0} = h_{*}(\cT)$, whose $(n-1)$-skeleton has measure zero. In addition to this, all the simplices of $\cT_{0}$ have the same mass. Using successive barycentric subdivisions we get a sequence $(\cT_{p})_{p \geq 0}$ of nested triangulations with both properties: the $(n-1)$-skeleton have no mass and all the simplices have the same mass. Denote by $m_{p}$ the common mass of the simplices of $\cT_{p}$, and by $d_{p}$ the supremum of the diameters of the simplices of $\cT_{p}$ (for some metric which is compatible with the topology on $M$).
Then  the sequences $(m_{p})$ and $(d_{p})$ tends to zero. 

Here is a useful consequence. Let $O$ be any open subset of $M$. We define inductively $\cO_{p}$ as the set of all the $n$-dimensional  open simplices of  $\cT_{p}$ that are included in $O$ but not in some $s \in \cO_{p-1}$. The elements of $\cO:= \cup \cO_{p}$ are pairwise disjoint and their closures cover $O$. Since the $(n-1)$-skeleton of our triangulations have no mass, we have the equality
$$
\mu(O) = \sum_{U \in \cO} \mu(U) \ \ \ (1).
$$
We call a (closed) simplex of some $\cT_{p}$ \emph{good} if it is a good ball in $M$. We notice that for every $p>0$, all the $n$-dimensional  simplices that are disjoint from the $(n-1)$-skeleton of $\cT_{0}$ are good\footnote{Note that there may be simplices in $\cT_{0}$ that fail to be good balls if $\cT_{0}$ is a triangulation but not a PL-triangulation.}. Thus equality (1) still holds if, in the definition of the $\cO_{p}$'s,  we replace the simplices by the simplices whose closure is good.
 As a consequence, if two probability measures $\mu,\mu'$ give the same mass to all the good simplices of $\cT_{p}$ for every $p$, then they are equal.

\bigskip

In the first Lemma we look for elements of the group $G_{f}$  that do not preserve the measure and have small support.
\begin{lemm}\label{lemma.small-support}
For every positive $\varepsilon$ there exists a good ball $B$ of measure less than $\varepsilon$ and an element $g \in G_{f}$ which is supported in $B$ and does not preserve the measure $\mu$.
\end{lemm}

\begin{proof}
By hypothesis the probability measures $\mu$ and $f_{*}\mu$ are not equal. According to the discussion preceding the Lemma, there exists some $p>0$ and some simplex of the triangulation $\cT_{p}$ whose closure $B_{1}$ is a good ball, and such that $\mu(B_{1}) \neq \mu(f^{-1}(B_{1}))$. 
To fix ideas let us assume that 
$$
\mu(f^{-1}(B_{1})) > \mu(B_{1}).
$$ 
This implies  the same inequality for at least one of the simplices of $\cT_{p+1}$ that are included in $B_{1}$ ; thus, by induction, we see that we may choose $p$ to be arbitrarily large.
Note that we have $\mu(f^{-1}(M \setminus B_{1})) < \mu(M \setminus B_{1})$. 
Thus the same reasoning, applied to $M \setminus B_{1}$, provides a (closed) simplex $B_{2}$ of some $\cT_{p'}$, disjoint from $B_{1}$, such that
$$
\mu(f^{-1}(B_{2})) < \mu(B_{2}).
$$ 
Again, by induction, we may assume that $p'=p$ and this is an arbitrarily large integer. In particular $B_{1}$ and $B_{2}$ are good balls with the same mass.
Let $B'$ be a ball whose interior contains $B_{1}$ and $B_{2}$.
Since $B_{1}$ and $B_{2}$ have the same measure, by the above mentioned version of the Oxtoby-Ulam theorem  there exists $\phi \in \homeo_{0}(M,\mu)$ supported in $B'$ and sending $B_{1}$ onto $B_{2}$.
Now we consider the element 
 $$
g = f^{-1} \phi f
 $$
of the group $G_{f}$. It has support in the ball $B=f^{-1}(B')$.
It sends the ball $f^{-1}(B_{1})$ to the ball $f^{-1}(B_{2})$, and we have
$$
\mu(f^{-1}(B_{1})) > \mu(B_{1}) = \mu(B_{2}) > \mu(f^{-1}(B_{2}))
$$
so that $g$ does not preserve the measure $\mu$, as required by the Lemma. 

It remains to see that in the above construction we may have chosen $B$ to be a good ball of arbitrarily small measure. Since $\mu$ has no atom, for every $\varepsilon>0$ there exists some $\eta>0$ such that every subset of $M$ of diameter less than $\eta$ has measure less than $\varepsilon$. Thus by choosing $p=p'$ large enough we may require that 
$$
\mu(f^{-1}(B_{1})) +  \mu(f^{-1}(B_{2})) < \varepsilon.
$$
Then we choose $B$ as a ball whose interior contains $f^{-1}(B_{1})$ and $f^{-1}(B_{2})$ and which still has measure less than $\varepsilon$. Finally we shrink $B$ a little bit to turn it into a good ball. This completes the proof of the Lemma.
\end{proof}

We subdivide the euclidean unit ball $B_{1}(0)$ of $\bbR^n$ into the half-balls $B_{1}^- = B_{1}(0) \cap \{x \leq 0\}$ and $B_{1}^+ = B_{1}(0) \cap \{x \geq 0\}$. Let $\Sigma$ be the disk $B_{1}^- \cap B_{1}^+$ that separates the half-balls.
We consider a given ball $B$ and some  homeomorphism $g$ supported in $B$.
For every homeomorphism  $\gamma : B_{1}(0) \to B$ we let $\gamma^\pm = \gamma(B_{1}^\pm)$; we say that  $\gamma$ is \emph{thin} if $\gamma(\Sigma)$ has measure zero.  
We now consider the set  $\cI(\gamma,g)$ of all the numbers of the type   
 $$\mu(g(\gamma^+)) - \mu(\gamma^+)$$
where  $\gamma$ is thin.
\begin{lemm}\label{lemma.interval}
If $g$ does not preserve the measure $\mu$ then  $\cI(\gamma,g)$ contains an interval $[a^-,a^+]$ with $a^- < 0 < a^+$.
\end{lemm}

\begin{proof}
First we want to prove that there exists some $\gamma:B_{1}(0) \to B$ which is thin and such that $\mu(g(\gamma^+)) \neq \mu(\gamma^+)$.
Since $g$ does not preserve the measure $\mu$, we may find some good ball $b$ in the interior of $B$ such that $\mu(b) \neq \mu(f^{-1}(b))$. To fix ideas we assume that 
$\mu(b) < \mu(f^{-1}(b))$. Thanks to the Oxtoby-Ulam theorem we may identify $B$ with a euclidean ball in $\bbR^n$, $b$ with another euclidean ball inside $B$, and $\mu$ with the restriction of the  Lebesgue measure on $\bbR^n$. All our balls are centered at the origin. Let $b'$ be a ball slightly greater than $b$, and $T$ be a thin tube in $B \setminus b'$ connecting the boundary of $B$ and that of $b'$. There exists a homeomorphism $\gamma : B_{1}(0) \to B$ such that $\gamma^+ = T \cup b'$.
The construction may be done so that the (Lebesgue) measure of $\gamma^+$ is arbitrarily close to that of $b$, and then we have $\mu(\gamma^+)) < \mu(g^{-1}(\gamma^+))$, as wanted.

We can find  a continuous family  $(R_{t})_{t\in [0,1]}$ of rotations of $B_{1}(0)$ such that $R_{0}$ is the identity and $R_{1}$ is a rotation that exchanges $B_{1}^-$ and $B_{1}^+$.  Setting $\gamma_{t}:=\gamma \circ R_{t}$, we have  $\gamma_{1}^+ = \gamma_{0}^- = \gamma^-$. Note that it may happen that $\gamma_{t}(\Sigma)$ has positive measure for some $t$. To remedy for this we consider $\gamma'=\phi \circ \gamma$, where $\phi: B \to B$ is a homeomorphism that fixes $\gamma(\Sigma)$, such that the image under $\gamma'$ of the Lebesgue measure on $B_{1}(0)$ is equivalent to the restriction of $\mu$ to the ball $B$, in the sense that both measures share the same measure zero sets; such a $\phi$ is provided by the Oxtoby-Ulam theorem. This ensures that $\gamma'_{t}:=\gamma' \circ R_{t}$ is thin for every $t$. Note that ${\gamma'_{0}}^\pm = \gamma_{0}^\pm$ and ${\gamma'_{1}}^\pm = \gamma_{1}^\pm$. We have
\begin{eqnarray*}
\mu(g({\gamma'_{1}}^+)) - \mu({\gamma'_{1}}^+) &  =  & \mu(g({\gamma'_{0}}^-)) - \mu({\gamma'_{0}}^-) \\
& =&  (1- \mu(g({\gamma'_{0}}^+))) - (1- \mu({\gamma'_{0}}^+)) \\
& = & - (\mu(g({\gamma'_{0}}^+)) - \mu({\gamma'_{0}}^+)) \neq 0.
\end{eqnarray*}
Thus the set $\cI(\gamma,g)$ contains the interval 
$$
\{\mu(g({\gamma'_{t}}^+)) - \mu({\gamma'_{t}}^+), t \in [0,1]\}
$$
which contains both a positive and a negative number, as required by the lemma.
\end{proof}

\begin{coro}\label{coro.transfer}
Let $\gamma_{0} : B_{1}(0) \to M$ be a topological embedding in $M$ with $\mu(\gamma_{0}(\Sigma))=0$, let $B_{0} = \gamma_{0}(B_{1}(0))$, and let $\varepsilon>0$ be less than the measure of  $\gamma_{0}^+$.
Then there exists some element $g \in G_{f}$, supported in $B_{0}$, such that
$$
\mu(g(\gamma_{0}^+)) = \mu(\gamma_{0}^+) - \varepsilon.
$$
\end{coro}
In the situation of the corollary we will say that $g$ \emph{transfers a mass $\varepsilon$ from $\gamma_{0}^+$ to $\gamma_{0}^-$}. 

\begin{proof}
Lemma~\ref{lemma.small-support} provides some element $g' \in G_{f}$ that does not preserve the measure $\mu$, and which is supported on a good ball $B$ whose measure is less than the minimum of $\mu(\gamma_{0}^+)-\varepsilon$ and $\mu(\gamma_{0}^-)$.
Then Lemma~\ref{lemma.interval} provides some homeomorphism $\gamma:B_{1}(0) \to B$ which is thin and such that 
$g'$ transfers some mass $a$ from $\gamma^+$ to $\gamma^-$:
$$
\mu(g'(\gamma^+))-\mu(\gamma^+) = a.
$$
Since such a number $a$ may be chosen freely in an open interval containing $0$, we may assume that $a = \frac{\varepsilon}{N}$ for some positive integer $N$.
Choose some homeomorphism $\Phi_{1} \in \homeo_{0}(M,\mu)$ that sends $B$ inside $B_{0}$, $\gamma^+$ inside $\gamma_{0}^+$ and $\gamma^-$ inside $\gamma_{0}^-$. Such a $\Phi_{1}$ is provided by Oxtoby-Ulam theorem, thanks to the fact that we have chosen the measure of $B$ to be small enough and that $\mu(\gamma(\Sigma)) = \mu(\gamma_{0}(\Sigma))=0$.
Now the conjugate $g_{1} = \Phi_{1} g' \Phi_{1}^{-1}$ transfers a mass $a$ from $\gamma_{0}^+$ to $\gamma_{0}^-$:
\begin{eqnarray*}
\mu(g_{1}(\gamma_{0}^+)) & = & \mu(\gamma_{0}^+)-a.
\end{eqnarray*}
We repeat the process with $\gamma_{1} = g_{1} \circ \gamma_{0}$ instead of $\gamma_{0}$, getting an element $g_{2} \in G_{f}$ that tranfers a mass $a$ from $\gamma_{1}^+$ to $\gamma_{1}^-$:
\begin{eqnarray*}
\mu(g_{2}g_{1}(\gamma_{0}^+)) & = & \mu(g_{2}(\gamma_{1}^+)) \\
& = & \mu(\gamma_{1}^+)-a \\
& = & \mu(g_{1}(\gamma_{0}^+))-a \\
& = & \mu(\gamma_{0}^+)-2a. \\
\end{eqnarray*}
We repeat the process $N$ times, and get the final homeomorphism $g$ as a composition  of the $N$ homeomorphisms $g_{N}, \dots , g_{1}$.
\end{proof}

\section{Proof of the theorem}

We consider as before some element $f \in \homeo_{0}(M) \setminus  \homeo_{0}(M,\mu)$. Let $g$ be some other element in $\homeo_{0}(M)$.
In order to prove the theorem  we want to approximate $g$ with some element in  the group $G_{f}$ generated by $f$ and $\homeo_{0}(M,\mu)$.
 We fix a triangulation $\cT_{0}$ for which the $(n-1)$-skeleton has zero measure. The first step of the proof consists in finding an element $g' \in G_{f}$ satisfying the following property:
\emph{for every simplex $s$ of $\cT_{0}$, the measure of $g'(s)$ coincides with the measure of $g^{-1}(s)$}. To achieve this, the (very natural) idea is to use corollary~\ref{coro.transfer} to progressively transfer some mass from the simplices $s$ whose mass is larger than the mass of their image under $g^{-1}$, to those for which the opposite holds.

Here are some details. Given a triangulation $\cT$  for which the $(n-1)$-skeleton has zero measure, we choose two $n$-dimensional  simplices $s, s'$ of $\cT$, and some positive $\varepsilon$ less than $\mu(s)$; let us explain how to transfer a mass $\varepsilon$ from $s$ to $s'$. First assume that $s$ and $s'$ are adjacent.
Then  we may choose an embedding $\gamma : B_{1}(0) \to s \cup s'$ with $\gamma(\Sigma) \subset s \cap s'$,  $\gamma^+ \subset s$ and $\gamma^- \subset s'$, and we apply corollary~\ref{coro.transfer}. Thus we get an  element $h \in G_{f}$, supported in $s \cup s'$, such that $\mu(h(s)) = \mu(s)-\varepsilon$, and consequently $\mu(h(s')) = \mu(s')+\varepsilon$. Now consider the general case, when $s$ and $s'$ are not adjacent. Since $M$ is connected, there exists a sequence $s_{0} = s, \dots , s_{\ell} = s'$ of simplices of $\cT$ in which two successive elements are adjacent. As described before we may transfer mass $\varepsilon$ from $s_{0}$ to $s_{1}$, then from $s_{1}$ to $s_{2}$, and so on. Thus by successive adjacent transfers of mass we get some element in $h \in G_{f}$ that transfers mass $\varepsilon$ from $s$ to $s'$. Note that the masses of all the other elements do not change, that is, $\mu(h(\sigma)) = \mu(\sigma)$ for every simplex $\sigma$ of $\cT$ different from $s$ and $s'$.

 Now we go back to our triangulation $\cT_{0}$, and we construct $g'$ the following way. If each simplex $s$ has the same measure as its inverse image $g^{-1}(s)$ then there is nothing to do. In the opposite case there exists some simplex $s$ of $\cT_{0}$ such that $\mu(s) > \mu(g^{-1}(s))$. We also select some other simplex $s'$ such that $\mu(s') \neq \mu(g^{-1}(s'))$, and we use the previously described construction of a homeomorphism $g_{1} \in G_{f}$  that transfers the mass $\mu(s)-\mu(g^{-1}(s))$ from the simplex $s$ to the simplex $s'$. After doing so the number of simplices $g_{1}(s) \in {g_{1}}_{*} \cT_{0}$ whose mass differs from the mass of $g^{-1}(s)$ has decreased by at least one compared to $\cT_{0}$. We proceed recursively until we get an element $g' \in G_{f}$ such that $\mu(g'(s)) = \mu(g^{-1}(s))$ for every simplex $s$ in $\cT_{0}$, as wanted for this first step.

For the second and last step we consider the triangulations $(g^{-1})_{*}(\cT_{0})$ and $g'_{*}(\cT_{0})$. The homeomorphism $g'g$ sends the first one to the second one, and each simplex $g^{-1}(s) \in (g^{-1})_{*}(\cT_{0})$ has the same measure as its image $g'(s) \in g'_{*}(\cT_{0})$. We apply Oxtoby-Ulam theorem independently on each $g'(s)$ to get a homeomorphism $\Phi_{s} : g'(s) \to g'(s)$, which is the identity on $\partial g'(s)$, and which sends the measure $(g'g)_{*} (\mu_{\mid g^{-1}(s)})$ to the measure $\mu_{\mid g'(s)}$. The homeomorphism
$$
\Phi:=  \left( \prod_{s} \Phi_{s} \right) g'g
$$
preserves the measure $\mu$. Furthermore by Alexander's trick each $\Phi_{s}$ is isotopic to the identity, thus $\Phi$ is isotopic to the identity, and belongs to the group  $ \homeo_{0}(M,\mu)$.
Now the homeomorphism $g''=g'^{-1}\Phi$ belongs to the group $G_{f}$ and for each  simplex $s$  of the triangulation $\cT_{0}$ we have
$g''^{-1}(s) = g^{-1}(s)$. We may have chosen the triangulation $\cT_{0}$ so that each simplex has diameter less than some given $\varepsilon$. Every point $x$ in $M$ belongs to some $n$-dimensional  closed simplex $g^{-1}(s)$ of the triangulation $(g^{-1})_{*}\cT_{0}$, and since both $g(x)$ and $g''(x)$ belong to $s$ they are a distance less than $\varepsilon$ apart. In other words the uniform distance from $g$ to $g''$ is less than $\varepsilon$. This proves that $g$ belongs to the closure of $G_{f}$, and completes the proof of the theorem.

\bibliographystyle{alpha}
\bibliography{bibliographie}

\end{document}